\documentclass{amsart}
\usepackage{amsmath,amsthm,amssymb}

\theoremstyle{definition}
\newtheorem{Thm}{Theorem}[section]

\newtheorem{Lem}[Thm]{Lemma}
\newtheorem{Cor}[Thm]{Corollary}

\newtheorem{Prop}[Thm]{Proposition}
\newtheorem{Rem}[Thm]{Remark}

\def\Stab{\mathrm{Stab}}
\def\Int{\mathrm{Int}}
\def\Fix{\mathrm{Fix}}
\def\Homeo{\mathrm{Homeo}}
\def\Ker{\mathrm{Ker}}

\def\N{\mathbb{N}}

\title[$C^*$-simplicity for convergence groups]
{$C^*$-simplicity for groups with non-elementary convergence group actions}
\author[Matsuda, Oguni, Yamagata]
{Yoshifumi Matsuda, Shin-ichi Oguni, Saeko Yamagata}
\address{Yoshifumi Matsuda, Graduate School of Mathematical Sciences, 
University of Tokyo, 
3-8-1 Komaba, 
Meguro-ku, 
Tokyo, 
153-8914 Japan}
\email{ymatsuda@ms.u-tokyo.ac.jp}
\address{Shin-ichi Oguni, Department of Mathematics, Faculty of Science,
Ehime University,
2-5 Bunkyo-cho, 
Matsuyama, 
Ehime, 
790-8577 Japan}
\email{oguni@math.sci.ehime-u.ac.jp}
\address{Saeko Yamagata, Akashi National College of Technology, 674-8501 Akashi, Japan}
\email{yamagata@akashi.ac.jp}

\begin{document}

\maketitle

\begin{abstract}
We prove that 
a countable group with an effective minimal non-elementary convergence group action 
is a Powers group. 
More strongly we prove that it is a strongly Powers group and thus 
its non-trivial subnormal subgroups
are $C^*$-simple.
\\


\noindent
2010MSC:
37B05;
20F65;
22D25.
\end{abstract}

\section{Introduction}\label{intro}
$C^*$-simplicity for countable groups have been studied very much 
since Powers proved that non-abelian free groups are $C^*$-simple
(\cite{Pow75}), where a countable group is $C^*$-simple 
if the reduced $C^*$-algebra of it is simple.
For example, it is known that torsion-free hyperbolic groups
which are not 
cyclic
are $C^*$-simple (\cite{Har88}). 
Also some criteria for $C^*$-simplicity for groups are established 
(refer to \cite{B-H}, \cite{Har} and \cite{H-P}).
In particular Powers groups are $C^*$-simple (\cite[Theorem 13]{Har}). 
See \cite[Definition 9]{Har} about the definition of Powers groups.

The following is our main theorem:
\begin{Thm}\label{P}
Let $G$ be a countable group. 
If $G$ has an effective minimal non-elementary convergence group action, 
then $G$ is a Powers group. 
\end{Thm}
\noindent
On the above the case where $G$ is torsion-free is known 
(\cite[Corollary 12 (iv)]{Har}). 
See \cite[Corollary 12]{Har} and \cite{H-P} for other examples of Powers groups. 

We give two corollaries.
First one is the following: 
\begin{Cor}\label{simple}
Let $G$ be a countable group 
with a minimal non-elementary convergence group action. 
Then the following are equivalent:
\begin{enumerate}
\item[(i)] the action is effective;
\item[(ii)] $G$ is a Powers group;
\item[(iii)] the reduced $C^*$-algebra of $G$ is simple;
\item[(iv)] the reduced $C^*$-algebra of $G$ has a unique normalized trace;
\item[(v)] $G$ has infinite conjugacy classes; 
\item[(vi)] $G$ does not have non-trivial amenable normal subgroups;
\item[(vii)] $G$ does not have non-trivial finite normal subgroups;
\item[(viii)] all minimal non-elementary convergence group actions of $G$ are effective.
\end{enumerate}
\end{Cor}
\noindent
See Section \ref{conv} about convergence group actions. 
Corollary \ref{simple} implies 
\cite[Corollary 2]{A-M}, which deals with the case of 
relatively hyperbolic groups in the sense of Osin. 
We remark that if a countable group 
$G$ is a properly relatively hyperbolic group in the sense of Osin
and is not virtually cyclic, then $G$ has a minimal non-elementary convergence group action
(see Section \ref{conv}). 

Recall that 
each group has a unique maximal amenable normal subgroup, 
which is called its amenable radical.
The second corollary is the following: 
\begin{Cor}\label{kernel'}
Let $G$ be a countable group 
with a minimal non-elementary convergence group action. 
Then the amenable radical $R_a(G)$ is finite and equal to the kernel of the action.
Also $G/R_a(G)$ is a strongly Powers group.
\end{Cor}
\noindent
Here a group is called a strongly Powers group when its non-trivial subnormal subgroups
are Powers groups (see \cite[1. Introduction]{H-P}).

It follows from Corollaries \ref{simple} and \ref{kernel'}
that a countable group 
with an effective minimal non-elementary convergence group action
is a strongly Powers group and thus 
its non-trivial subnormal subgroups are $C^*$-simple (see Corollary \ref{Powers'}).

\section{Properties of convergence group actions}\label{conv}
The study of convergence groups was initiated in \cite{G-M87}. 
In this section we recall some definitions and properties related to convergence group actions
(refer to \cite{Tuk94}, \cite{Fre97} and \cite{Bow99a}).

Let $G$ be a countable group, 
$X$ be a compact metrizable space and 
$\rho:G\to \Homeo(X)$ be a homomorphism. 
The pair $(\rho,X)$ is a convergence group action if 
$X$ has at least three points, 
and for any infinite sequence $\{g_i\}$ of mutually different elements of $G$, 
there exist a subsequence $\{g_{i_j}\}$ of $\{g_i\}$
and two points $r,a\in X$ such that 
$\rho(g_{i_j})|_{X\setminus \{r\}}$ converges to $a$ uniformly on compact subsets of 
$X\setminus \{r\}$ and also 
$\rho(g_{i_j}^{-1})|_{X\setminus \{a\}}$ converges to $r$ uniformly on compact subsets of 
$X\setminus \{a\}$. 
The sequence $\{g_{i_j}\}$
is called a convergence sequence of $(\rho, X)$
and also the points $r$ and $a$ are called 
the repelling point of $\{g_{i_j}\}$ and 
the attracting point of $\{g_{i_j}\}$, respectively.
When we consider a convergence group action $(\rho,X)$, 
for $l\in G$, we call $\rho(l)$ a loxodromic element
if $\rho(l)$ is of infinite order
and has exactly two fixed points.
The sequence $\{l^i\}_{i\in \N}$ is a convergence sequence of $(\rho,X)$
with the repelling point $r$ and the attracting point $a$, 
which are distinct and fixed by $\rho(l)$.
Hence we call $r$ the repelling fixed point of $\rho(l)$ and 
$a$ the attracting fixed point of $\rho(l)$.

Let $(\rho,X)$ be a convergence group action of a countable group $G$.
Since $X$ has at least three points, 
$\Ker\rho$ has no convergence sequences 
and thus $\Ker\rho$ is finite. 
The set of all repelling points and attracting points 
is equal to the limit set $\Lambda(\rho)$
(\cite[Lemma 2M]{Tuk94}). 
The cardinality of $\Lambda(\rho)$ is $0$, $1$, $2$ or $\infty$ 
(\cite[Theorem 2S, Theorem 2T]{Tuk94}). 
We remark that $\#\Lambda(\rho)=0$ if $G$ is finite by definition. 
Also it is well-known that $G$ is virtually infinite cyclic
if $\#\Lambda(\rho)=2$ 
(see \cite[Lemma 2Q,Lemma 2N,Theorem 2I]{Tuk94} and also \cite[Example 1.3]{Fre97}).
We call $(\rho, X)$ a non-elementary convergence group action  
if $\#\Lambda(\rho)=\infty$. 
We note that if a countable group $G$ has 
a non-elementary convergence group action, 
then the induced action on the limit set
is a minimal non-elementary convergence group action.

We briefly recall the notion of relatively hyperbolic group,
which was introduced in \cite{Gro87}.
Let $G$ be a countable group with a finite family of infinite subgroups
$\mathbb H$. Assume that $G$ is not virtually cyclic for simplicity. 
We call $(G,{\mathbb H})$ a properly relatively hyperbolic group if 
there exists a geometrically finite convergence group action $(\rho,X)$
such that $\mathbb H$ is a set of representatives of conjugacy classes
of maximal parabolic subgroups 
(refer to \cite[Definition 1]{Bow97}, \cite[Theorem 0.1]{Yam04} 
and \cite[Definition 3.1]{Hru10} for details). 
We remark that geometrically finite convergence group actions are 
minimal and non-elementary. 
Also we note that 
several definitions of relative hyperbolicity for a pair $(G,{\mathbb H})$
of a countable group and a family of subgroups 
are mutually equivalent 
if $\mathbb H$ is finite and each $H\in \mathbb H$ is infinite (refer to \cite{Hru10}). 

\begin{Rem}\label{rem}
If we consider a relatively hyperbolic group $(G,{\mathbb H})$
in the sense of Osin (\cite[Definition 2.35]{Osi06}), then 
$\mathbb H$ can be infinite
and each $H\in \mathbb H$ can be finite.
However even if $G$ is hyperbolic relative to an infinite family $\mathbb H$ 
in the sense of Osin, 
then $G$ can be realized as a free product of two infinite subgroups $A$ and $B$
(\cite[Theorem 2.44]{Osi06}), 
and thus $G$ is hyperbolic relative to $\{A, B\}$, which is a finite family of infinite subgroups.  
\end{Rem}

\begin{Rem}\label{rem'}
We do not know whether there exists a countable group $G$ such that 
it is not hyperbolic relative to any family of proper subgroups
in the sense of Osin, 
but it has a non-elementary convergence group action. 
\end{Rem}

We need two facts in order to prove Corollary \ref{kernel'}.
First one is claimed in \cite[Section 1]{Fre97}.
\begin{Prop}\label{ker}
Let $G$ be a countable group 
with a minimal non-elementary convergence group action $(\rho, X)$.
Then $\Ker \rho$ is the maximal finite normal subgroup of $G$.
\end{Prop}
\noindent
We give its proof for readers in Appendix \ref{a}.\\
Second one is \cite[Proposition 3.1]{Fre97}. 
We give a proof for readers here.
\begin{Prop}\label{normal}
Let $G$ be a countable group 
with a convergence group action $(\rho, X)$. 
Let $N$ be an infinite normal subgroup of $G$. 
We consider the restricted convergence group action $(\rho|_{N},X)$. 
Then we have $\Lambda(\rho|_{N})=\Lambda(\rho)$.
In particular if $(\rho, X)$ is non-elementary, then
$(\rho |_N, X)$ is also non-elementary.
\end{Prop}
\begin{proof}
We remark that $\Lambda(\rho |_N)$ is not empty since $N$ is infinite.
Also obviously we have $\Lambda(\rho |_N)\subset \Lambda(\rho)$.  
Hence we have $\Lambda(\rho |_N)=\Lambda(\rho)$ if $\#\Lambda(\rho)=1$.

If $\#\Lambda(\rho)=2$, then $G$ is virtually infinite cyclic.
Then $N$ is also virtually infinite cyclic. 
When we take an element $n\in N$ of infinite order, 
$\rho(n)$ is loxodromic and fixes $\Lambda(\rho)$ 
(\cite[Lemma 2Q, Theorem 2G]{Tuk94}). 
Thus we have $\Lambda(\rho |_N)\supset \Lambda(\rho)$. 

Suppose that $(\rho,X)$ is non-elementary. 
Moreover we can assume that it is minimal without loss of generality. 
We take a convergence sequence $\{n_i\}_{i\in \N}$ of $(\rho |_N,X)$
with the attracting point $a\in X$ and 
the repelling point $r\in X$. 
Then for any $g\in G$, 
the infinite sequence
$\{gn_ig^{-1}\}_{i\in \N}$ is a convergence sequence of $(\rho,X)$ 
with the attracting point $\rho(g)a\in X$ and 
the repelling point $\rho(g)r\in X$. 
The infinite sequence
$\{gn_ig^{-1}\}_{i\in \N}$ can be regarded as a convergence sequence of 
$(\rho |_N,X)$ since $N$ is normal. 
In particular we have $\rho(G)a\subset\Lambda(\rho |_N)$.
Moreover the closure $\overline{\rho(G)a}$ of 
$\rho(G)a$ in $X$ is contained in $\Lambda(\rho |_N)$
since $\Lambda(\rho |_N)$ is closed.
Since $\Lambda(\rho)=\overline{\rho(G)a}$ (\cite[Theorem 2S]{Tuk94}),
we have $\Lambda(\rho |_N)\supset \Lambda(\rho)$.
\end{proof}

\section{Proofs of results}
In this section we prove Theorem \ref{P} and Corollaries \ref{simple} and \ref{kernel'}.

We show the following in order to prove Theorem \ref{P}:
\begin{Prop}\label{freept}
Let $G$ be a countable group 
with a minimal non-elementary convergence group action $(\rho,X)$.
Then 
$(\rho, X)$ is effective 
if and only if 
there exists 
a point $x\in X$ such that $\rho(g)x\neq x$ for any element $g\in G\setminus \{1\}$.
\end{Prop}
\noindent
This claims that effectiveness of a minimal non-elementary convergence group action  
can be detected by some single point.

The following is a key lemma for Proof of Proposition \ref{freept}.
\begin{Lem}\label{noint}
Let $G$ be a countable group 
with a minimal non-elementary convergence group action $(\rho,X)$.
Then for any element $g\in G$ such that $\rho(g)$ is not the identity map, 
the fixed point set $\Fix(\rho(g))\subset X$ has an empty interior. 
\end{Lem}
\begin{proof}
We take an element $g\in G$ such that $\rho(g)$ is not the identity map.
We put $U:=X\setminus\Fix(\rho(g))$, which is clearly open and not empty.  
We assume that the interior $V$ of $\Fix(\rho(g))$ is not empty.
Since $U$ and $V$ are mutually disjoint open sets which are not empty, 
we have an element $l\in G$ such that $\rho(l)$ is loxodromic, 
a fixed point $r$ of $\rho(l)$ is in $U$
and the other fixed point $a$ of $\rho(l)$ is in $V$
(\cite[Theorem 2R]{Tuk94}). 
Then $\rho(glg^{-1})$ is loxodromic and 
satisfies $\Fix(\rho(glg^{-1}))=\{\rho(g)r,\rho(g)a\}$. 
Since $r\in U$ and $a\in V$, we have $\rho(g)r\neq r$ and $\rho(g)a=a$.
This contradicts the fact
that the fixed point sets of two loxodromic elements
either coincide or have an empty intersection
(\cite[Theorem 2G]{Tuk94}). 
\end{proof}

\begin{proof}[Proof of Proposition \ref{freept}]
The `if' part is trivial. We prove the `only if' part.
Suppose that $(\rho, X)$ is effective. 
Then for each $g\in G\setminus \{1\}$, 
$\rho(g)$ is not the identity map.
Hence the interior $\Int(\Fix(\rho(g)))$ is empty by Lemma \ref{noint}.
Then $\Int(\bigcup_{g\in G\setminus \{1\}}\Fix(\rho(g)))$ is also empty since 
$X$ is a Baire space. 
Thus $\bigcup_{g\in G\setminus \{1\}}\Fix(\rho(g))$ must be a proper subset of $X$
because $X$ is not empty.
\end{proof}

\begin{Rem}
Since Lemma \ref{noint} claims that 
an effective minimal non-elementary convergence group action is 
`slender' (see the paragraph just before \cite[Corollary 10]{H-P}), 
the argument of the proof of Proposition \ref{freept} is parallel to 
an observation in the proof of \cite[Corollary 10]{H-P}.
\end{Rem}

\begin{proof}[Proof of Theorem \ref{P}]
Suppose that 
$G$ has an effective minimal non-elementary convergence group action $(\rho,X)$. 
There exist two elements $l_1,l_2\in G$ such that $\rho(l_1)$ and $\rho(l_2)$ are loxodromic
and have no common fixed points (\cite[Theorem 2T]{Tuk94}). 
Since $\{l_1^i\}_{i\in \N}$ and $\{l_2^i\}_{i\in \N}$ are convergence sequences, 
then $\rho(l_1)$ and $\rho(l_2)$ are hyperbolic in the sense of \cite[Definition 10]{Har}. 
Therefore $(\rho,X)$ is strongly hyperbolic in the sense of \cite[Definition 10]{Har}. 
Also for any finite subset $F\subset G\setminus\{1\}$, 
there exists $x\in X$ such that $\rho(f)x\neq x$ for all $f\in F$ 
by Proposition \ref{freept}.
Therefore $(\rho,X)$ is strongly faithful in the sense of \cite[Definition 10]{Har}. 
Hence $G$ is a Powers group by \cite[Proposition 11 and the following remark]{Har}.
\end{proof}

\begin{proof}[Proof of Corollary \ref{simple}]
Clearly (viii) implies (i). Also Theorem \ref{P} claims that 
(i) implies (ii).

For general countable groups, 
the following relations hold among properties (ii), (iii), (iv), (v), (vi), (vii) and (viii): 
It is well-known that 
(ii) implies both (iii) and (iv)
(see \cite[Theorem 13 and Remark (i) on it]{Har});
each of (iii) and (iv) implies both (v) and (vi) 
(see \cite[Proposition 3]{B-H}); 
each of (v) and (vi) obviously implies (vii); 
(vii) implies (viii) by definition of convergence group actions
(see Section \ref{conv}).
\end{proof}

We have the following: 
\begin{Cor}\label{Powers'}
Let $G$ be a countable group 
with an effective minimal non-elementary convergence group action $(\rho,X)$. 
Then for every non-trivial subnormal subgroup $N$ of $G$, the restricted action 
$(\rho |_N,X)$ is an effective minimal non-elementary convergence group action. 
In particular $G$ is a strongly Powers group.
\end{Cor}
\begin{proof}
We take a non-trivial subnormal subgroup $N$ of $G$. 
There exists a finite chain of subgroups $N=N_k<N_{k-1}<\cdots<N_0=G$
such that $N_j$ is normal in $N_{j-1}$ for any $j=1,\ldots,k$.  
We suppose that $G$ has an effective minimal non-elementary convergence group action 
$(\rho, X)$. 
Then $N_1$ is infinite by Proposition \ref{ker}
and thus the restricted action $(\rho|_{N_1},X)$ is 
an effective minimal non-elementary convergence group action 
by Proposition \ref{normal}. 
By induction on $j$, the restricted action $(\rho|_N,X)$ is 
an effective minimal non-elementary convergence group action.
Hence $N$ is a Powers group by Theorem \ref{P}.
\end{proof}

\begin{proof}[Proof of Corollary \ref{kernel'}]
Suppose that 
$G$ has a minimal non-elementary convergence group action $(\rho,X)$. 
$\Ker \rho$ is the maximal finite normal subgroup of $G$ 
by Proposition \ref{ker}. 
Now we prove that the amenable radical $R_a(G)$ is finite. 
We assume that $G$ has an infinite amenable normal subgroup $N$. 
Since $N$ does not contain any non-abelian free subgroups, 
the restricted convergence group action $\rho|_{N}$ 
is elementary (\cite[Theorem 2U]{Tuk94}).
This contradicts Proposition \ref{normal}. 

Since $\Ker \rho$ is equal to $R_a(G)$, 
the quotient $G/R_a(G)$ has 
the induced effective minimal non-elementary convergence group action $(\bar{\rho},X)$.
Hence Corollary \ref{Powers'} can be applied to $G/R_a(G)$.
\end{proof}

\appendix
\section{Proof of Proposition \ref{ker}}\label{a}
In this appendix we prove Proposition \ref{ker}.
In fact we prove Proposition \ref{kernel} 
(compare with \cite[Lemma 3.3]{A-M-O} for the case of relatively hyperbolic groups).

Let $G$ be a countable group 
with a convergence group action $(\rho,X)$
and $l$ be an element of $G$ such that $\rho(l)$ is loxodromic.
We put $E^+_{\rho}(l):=\Stab_{\rho}(r)\cap \Stab_{\rho}(a)$
and $E_{\rho}(l):=\Stab_{\rho}(\{r,a\})$, where 
$\Stab_{\rho}(r)$, $\Stab_{\rho}(a)$
and $\Stab_{\rho}(\{r,a\})$ are 
the stabilizer of subsets $\{r\}$, $\{a\}$ and $\{r,a\}$ in $G$, respectively.
Also we define $E_{\rho}^+(G)$ (resp. $E_{\rho}(G)$) 
by the intersection of all the sets in the family $\{ E^+_{\rho}(l) ~|~ l
\in G, \rho(l)\text{ is loxodromic}\}$
(resp.  $\{ E_{\rho}(l) ~|~ l \in G, \rho(l)\text{ is loxodromic}\}$). 

\begin{Prop}\label{kernel}
Let $G$ be a countable group 
with a non-elementary convergence group action.
If $(\rho, X)$ is a minimal non-elementary convergence group action, then 
$\Ker \rho$ is the maximal finite normal subgroup of $G$
and equal to both $E_{\rho}^+(G)$ and $E_{\rho}(G)$.
\end{Prop}

In order to prove the above, we need some lemmas.
\begin{Lem}\label{23}
Let $G$ be a countable group 
with a non-elementary convergence group action $(\rho,X)$.
Then $E_{\rho}^+(G)$ is a finite normal subgroup of $G$.
\end{Lem}
\begin{proof} 
For any element $l\in G$ such that $\rho(l)$ is loxodromic and for any element $h\in G$, 
$\rho(hlh^{-1})$ is loxodromic.
Clearly we have $E_{\rho}^+(hlh^{-1})=hE_{\rho}^+(l)h^{-1}$. 

We can take two elements $l_1,l_2\in G$ such that 
$\rho(l_1)$ and $\rho(l_2)$ are loxodromic and have no common fixed points 
(\cite[Theorem 2R]{Tuk94}). 
Then $E_{\rho}^+(l_1)\cap E_{\rho}^+(l_2)$ is finite because 
$E_{\rho}^+(l_1)$ and $E_{\rho}^+(l_2)$ have no common elements of infinite order (\cite[Theorem 2G]{Tuk94})
and they are virtually infinite cyclic by \cite[Theorem 2I]{Tuk94}. 
\end{proof}

We can prove the following in the same way as the proof of Lemma \ref{23}.
\begin{Lem}\label{23'}
Let $G$ be a countable group 
with a non-elementary convergence group action $(\rho,X)$.
Then $E_{\rho}(G)$ is a finite normal subgroup of $G$.
\end{Lem}

\begin{Lem}\label{4'}
Let $G$ be a countable group 
with a non-elementary convergence group action $(\rho,X)$, 
$l\in G$ be an element such that $\rho(l)$ is loxodromic 
and $g$ be an element of $G$. 
If there exists a positive integer $n$ such that $gl^ng^{-1}=l^n$, then 
$g$ is an element of $E_{\rho}^+(l)$.
\end{Lem}
\begin{proof}
We take the repelling fixed point $r$ of $\rho(l)$ and 
the attracting fixed point $a$ of $\rho(l)$. 
Then $\rho(gl^ng^{-1})$ is a loxodromic element with 
the repelling fixed point $\rho(g)r$ and 
the attracting fixed point $\rho(g)a$.
$gl^ng^{-1}=l^n$ implies that $\rho(g)r=r$ and $\rho(g)a=a$.
\end{proof}

\begin{Lem}\label{4}
Let $G$ be a countable group 
with a non-elementary convergence group action $(\rho,X)$.
Then any finite normal subgroup $M$ of $G$ is contained in $E_{\rho}^+(G)$.
\end{Lem}
\begin{proof}
When we consider a finite normal subgroup $M$ of $G$, 
we have $[G: C_G(M)]<\infty$, where $C_G(M)$ is the centralizer of $M$ in $G$. 
Hence for any element $l\in G$ such that $\rho(l)$ is loxodromic, 
there exists a positive integer $n$ such that $l^n\in C_G(M)$, 
that is, $ml^nm^{-1}=l^n$ for any $m\in M$. 
Thus we have $m\in E_{\rho}^+(l)$ by Lemma \ref{4'}.
\end{proof}

\begin{proof}[Proof of Proposition \ref{kernel}]
$E_{\rho}^+(G)$ is the maximal finite normal subgroup by Lemma \ref{23} and Lemma \ref{4}. 

Since $E_{\rho}(G)$ is a finite normal subgroup of $G$ by Lemma \ref{23'}, 
we have $E_{\rho}(G)\subset E_{\rho}^+(G)$.
Also we have $E_{\rho}(G)\supset E_{\rho}^+(G)$ by definition.

We take an element $l\in G$ such that $\rho(l)$ is loxodromic
and the repelling fixed point $r\in X$ of $\rho(l)$.
Then for any $g\in E_{\rho}^+(G)$, $\rho(g)$ 
fixes every point of the orbit $\rho(G)r$. 
Since $(\rho, X)$ is non-elementary and minimal, $\rho(g)$ 
fixes every point of $X$ by \cite[Theorem 2S]{Tuk94}. 
Thus we have $\Ker \rho\supset E_{\rho}^+(G)$.
Also we have $\Ker \rho\subset E_{\rho}^+(G)$ by definition. 
\end{proof}
\ \\

\noindent
\textbf{Acknowledgements.}
The authors would like to thank Professor Pierre de la Harpe
for giving them useful comments and informing them of \cite{H-P}.

The first author is supported by the Global COE Program 
at Graduate School of Mathematical Sciences, the University of Tokyo, 
and Grant-in-Aid for Scientific Researches for Young Scientists (B) (No. 22740034), 
Japan Society of Promotion of Science.

\end{document}